\documentclass[11pt]{amsart}

\usepackage{my_preamble}
\usepackage{fullpage}
\usepackage{color}
\usepackage[initials]{amsrefs}  
\theoremstyle{remark}

\newcommand{\aas}{\mathfrak{a}}
\newcommand{\bbs}{\mathfrak{b}}
\newcommand{\ccs}{\mathfrak{c}}

\newcommand{\sto}{\rightsquigarrow}
\newcommand{\EM}{\textsf{EM}}
\newcommand{\emb}{\mathrm{Emb}}
\renewcommand{\qtp}{\mathrm{qftp}}

\newcommand{\concat}{\widehat{\,\,}}

\newcommand{\FF}{\mathscr{F}}

\newcommand{\UU}{\mathbb{U}}

\newcommand{\fcn}{\mathrm{fn}}

\newcommand{\eqrk}{\textsf{eq-rk}}
\newcommand{\Fraisse}{\textrm{Fra\"iss\'e}}
\newcommand{\SX}{\textsf{X}}
\newcommand{\SF}{\textsf{F}}

\newcommand{\opD}{\textsf{opD}}
\newcommand{\MLOn}{ T_{n\mathrm{-mlo}} }
\newcommand{\Lnmlo}{ \mathscr{L}_{n\mathrm{-mlo}} }
\newcommand{\FFn}{ \mathscr{F}_{n\mathrm{-mlo}} }
\newcommand{\G}{\mathcal{G}}
\newcommand{\V}{\mathcal{V}}
\newcommand{\E}{\mathcal{E}}

\newcommand{\ms}{\vspace{.1in}}

\newcommand{\ceq}{\textrm{ceq}}

\newcommand{\leftexp}[2]{{\vphantom{#2}}^{#1}{#2}}
\newcommand{\lx}{<_{\textrm{lex}}}

\begin{document}


\title{Characterizing Model-Theoretic Dividing Lines via Collapse of Generalized Indiscernibles}
\author{Vincent Guingona}
\author{Cameron Donnay Hill}
\author{Lynn Scow}

\begin{abstract}
 We use the notion of collapse of generalized indiscernible sequences to classify various model theoretic dividing lines.  In particular, we use collapse of $n$-multi-order indiscernibles to characterize op-dimension $n$; collapse of function-space indiscernibles (i.e. parameterized equivalence relations) to characterize rosy theories; and finally, convex equivalence relation indiscernibles to characterize NTP2 theories.
\end{abstract}

\maketitle

\section{Introduction}

In model theory, and in S. Shelah's classification theory in particular, one central program is the search for robust dividing lines among complete theories -- dividing lines between intelligibility and non-structure.
 For a dividing line to be sufficiently interesting, one often desires that both sides of the line have interesting mathematical content.  Moreover, if such a dividing line has multiple characterizations coming from seemingly different contexts, this lends credence to the notion that the line is indeed important.  The exemplar of this in model theory is the notion of stability, first introduced by S. Shelah \cite{ShelahBook}.  Both stable and unstable theories are inherently interesting, and stability enjoys many different characterizations, from cardinalities of Stone spaces, to coding orders, to collapse of indiscernible sequences to indiscernible sets.  Another example of such a robust dividing line is NIP; like stability, there are important structures with theories on both sides of the line, and the model theory on both sides can be very rich.

In this paper, we focus on the notion of collapsing generalized indiscernibles as a means of  characterizating/defining  various model-theoretic dividing lines that are already well-established in the literature.  It can be said that this work started with Shelah's original characterization of stable theories, Theorem II.2.13 of \cite{ShelahBook}.  There he shows that a theory is stable if and only if every indiscernible sequence is an indiscernible set (i.e., the order of the sequence ``does not matter'' or is invisible to models of the theory).  In a similar fashion, it can be shown that other well-known model-theoretic dividing lines can be characterized similarly by such  ``collapse'' statements, which we formalize here.  In order to carry this out, we must expand our definition of ``indiscernibility.''  A quintessential example of this phenomenon beyond stability is the third author's characterization of NIP theories by ordered-graph indiscernibles.  In \cite{ScowNIP}, she shows that a theory is NIP if and only if any indiscernible ``picture'' of the generic ordered graph is actually indiscernible {\em without} the graph structure (i.e., it collapses to order in the sense the the graph relation is invisible to models of the theory).

The project in this paper was also alluded to in the work of the first and second authors on op-dimension \cite{GuinHillOP}, and this will be discussed further in Section \ref{Sect_opDim} below.  Generalized indiscernibles were first introduced by Shelah (see Section VII.2 of \cite{ShelahBook}), where they were used in the context of tree-indexed indiscernibles in the hope of understanding the tree property (see, for example, Theorem III.7.11 of \cite{ShelahBook}), but propose a slightly different formulation suggested by the second author in presently-unpublished work -- see \cite{HillIndisc} -- which slightly simplifies some aspects of Section \ref{Sect_Rosiness}.  For more on tree indiscernibles in particular, see \cites{KimKim, KimKimScow}.

\medskip

\textbf{Outline of the paper}.  In Section \ref{Sect_Indisc}, we introduce all the relevant notation and notions for generalized indiscernibility, the Ramsey Property, and the Modeling Property.  As noted above, this will mostly follow \cite{HillIndisc}.  In the remaining sections, we will apply this to specific cases, exhibiting the ``collapse'' characterizations of various dividing lines.  In Section \ref{Sect_opDim}, we begin with the example of $n$-multi-order indiscernibles, first considered in \cite{GuinHillOP}, and show that collapse down to $n$ orders characterizes op-dimension $n$ (Theorem \ref{Thm_opDimIndisc} below).  In Section \ref{Sect_Rosiness}, we consider function-space indiscernibles, and show that collapse down to two indiscernible sequences characterizes rosiness (Theorem \ref{thm:collapse-result-rosiness} below).  In Section \ref{Sect_NTP2}, we consider convex equivalence relation indiscernibles, showing that NTP2 is equivalent to a dichotomy between collapsing down to an indiscernible sequence or having dividing witnessed by a particular formula (Theorem \ref{Thm_NTP2} below).



\section{Theories of (generalized) indiscernibles}\label{Sect_Indisc}

In this section, we review the definitions associated with structural Ramsey theory: \Fraisse-like theories, the Ramsey Property, and the Modeling Property, and theories of indiscernibles. The development (taking up subsections 2.0.1 through 2.2) is just a summary of that in \cite{HillIndisc}, which in turn is based in large part on \cite{ScowNIP}. In Subsection 2.4, we formalize the notion of ``collapse of indiscernibles'' as it will be used in this article.

\subsubsection{Notation and conventions for finite structures}

Mirroring the notation used in \cite{HillIndisc}, we adopt a somewhat eccentric streamlined notation for finite structures.
\begin{defn}[Eccentric notation for finite structures]
 In order to save ourselves from writing ``...is a finite structure'' and ``is a finite substructure of'' {\em ad nauseum}, we establish a slightly eccentric convention. With the one exception of $\ell$-embeddings (see the next subsection), upper-case calligraphic letters -- $\A$, $\B$ and so forth -- always denote {\em infinite} structures. On the other hand, lower-case gothic letters -- $\aas,\bbs,\ccs$ and so forth -- always denote {\em finite} structures. Thus, $\aas\leq\A$, for example, means that $\aas$ is a finite (induced) substructure of $\A$. 
Moreover, given a 1-sorted structure $\aas$, $\|\aas\|$ is its universe (a finite set), and $|\aas|$ is the cardinality of $\aas$.

We write $\emb(\aas,\A)$, $\emb(\aas,\bbs)$, $\emb(\A,\B)$ for the sets of all embeddings $\aas\to\A$, $\aas\to\bbs$ and $\A\to\B$, respectively. For an embedding $u$ in one of these sets, $u\aas$ (or $u\A$) is the substructure of the codomain induced on the image of $u$. 
\end{defn}

\begin{rem}[Language Convention]
Although the machinery of theories of indiscernibles accomodates countably infinite relational languages (even allowing finitely many constant symbols), in this article, the signature $\mathrm{sig}(\L_0)$ of every language of a \Fraisse-like theory $(T_0,\FF_0)$ will be finite relational.
\end{rem}

\subsection{\Fraisse-like theories}

Since our languages will never have function symbols, it is necessary to accommodate ``base theories'' that are not universal.  To do this, we define the notion of a fragment -- a surrogate for the quantifier-free formulas -- and the notion of a \Fraisse-like theory, which is analogous to the universal theory of a genuine \fraisse class.
A theory in a language with a function symbol that is known to be a theory of indiscernibles can be converted to a purely relational theory over a base theory where the base theory states that the new relation symbol behaves as the old function symbol did.

\begin{defn}[Fragments]
Given a language $\L_0$, we define a {\em fragment} of $\L_0$ to be a subset $\FF$ of $\L_0$-formulas with the following closure properties:
\begin{itemize}
\item $\FF$ contains every quantifier-free $\L_0$-formula.
\item $\FF$ is closed under taking boolean combinations and under changes of free variables.
\item $\FF$ is closed under taking subformulas.
\end{itemize}
We write $\forall(\FF)$ for the set of formulas of the form $\forall \xx\,\phi$ where $\phi\in\FF$.  Abusing the notation somewhat, we also write $\forall(\FF)$ for the closure of this set under conjunctions and disjunctions.

We say that $\FF$ is the {\em trivial fragment} if it is the set of quantifier-free formulas of $\L_0$.
\end{defn}

\begin{defn}[\Fraisse-like theories]
Given a language $\L_0$, a \emph{\Fraisse-like theory} is a pair $(T_0,\FF)$, where $\FF$ is a fragment of $\L_0$ and $T_0$ is a set of $\forall(\FF)$-sentences, satisfying the following conditions:
\begin{itemize}
\item[(JEP)] For all $\aas_0,\aas_1\models T_0$, there is a $\bbs\models T_0$ such that $\emb_\FF(\aas_i,\bbs)\neq\emptyset$ ($i<2$).
\item[(AP)] For $\aas,\bbs_0,\bbs_1\models T_0$ and $f_i\in\emb_\FF(\aas,\bbs_i)$ ($i<2$), there are $\ccs\models T_0$ and $g_i\in\emb_\FF(\bbs_i,\ccs)$ ($i<2$) such that $g_1f_1 = g_0f_0$.
\item[(LS)] There is a function $\lambda:\omega\to\omega$ such that for any $\B\models T_0$ and $X\subseteq_\fin B$, there is an $\aas\preceq_\FF\B$ such that $X\subseteq\|\aas\|$ and $|\aas|\leq \lambda(|X|)$.
\end{itemize}
\end{defn}

As with \fraisse classes and \Fraisse-limits, \Fraisse-like theories support countably infinite generic models. The following proposition summarizes this fact.

\begin{prop}\label{Prop_comp}
Let $(T_0,\FF)$ be a \Fraisse-like theory in $\L_0$. Uniquely up to isomorphism, there is a countable $\L_0$-structure $\A$ such that:
\begin{itemize}
\item For all $\aas\models T_0$, $\emb_\FF(\aas,\A)\neq\emptyset$.
\item For any $\aas,\bbs\models T_0$ with $\aas\preceq_\FF\bbs$, for any $f_0\in\emb_\FF(\aas,\A)$, there is an $f\in \emb_\FF(\bbs,\A)$ such that $f_0\subseteq f$.
\end{itemize}
Moreover, $T_0^* :=Th(\A)$ is $\aleph_0$-categorical and model-complete relative to $\FF$. 
\end{prop}

\begin{defn}[$\FF$-model-completion; generic models]\label{Defn_comp}
In the notation of Proposition \ref{Prop_comp}, For a \Fraisse-like theory $(T_0,\FF)$, $T_0^*$ is the {\em $\FF$-model-completion of $T_0$}. 

The structure $\A \vDash T_0^*$ is called the {\em generic model} of $(T_0,\FF)$.
\end{defn}

\begin{rem}
In the case that $\FF$ is the trivial fragment, $T_0^*$ is the usual model-completion of $T_0$.
\end{rem}

Here we give a list of all \Fraisse-like theories $(T_0,\FF)$ discussed in the paper. 

\begin{defn}[Index of theories]\label{Defn_index} \
\begin{itemize}
\item $T_0$: stands for an arbitrary index theory
\item $T_n$, $0<n<\omega$: 
	\begin{itemize}
	\item Let $\L_n$ be the $n$-sorted language with sorts $\SX_0,...,\SX_{n-1}$ and for each $i<n$, a binary relation symbol $<_i$ on $\SX_i$. Let $T_n$ be the $\L_n$-theory asserting that each $<_i$ is a linear ordering of $\SX_i$. 
	\item  Let $\FF_n$ be the trivial fragment.
	\end{itemize}
\item $T_1$: this is the theory of linear order (see above)
\item $T_{n-\mathrm{mlo}}$, $0<n \leq \omega$: 
	\begin{itemize}
	\item Let $\Lnmlo = \{ <_0, ..., <_{n-1} \}$ be the $1$-sorted language with $n$ binary relations.  Let $\MLOn$ be the $\Lnmlo$-theory which asserts that $<_i$ is a linear order for each $i < n$.  
	\item Let $\FFn$ to be the trivial fragment.
	\end{itemize}
\item $T_\fcn$, $T_{\textrm{eq:n}}$, $0<n \leq \omega$: 
	\begin{itemize}
	\item Let $\L_\fcn$ be the two-sorted language with sorts $\SX$ and $\SF$, two binary relation symbols $<_\SX$ and $<_\SF$ on $\SX$ and $\SF$, respectively, and one more relation symbol $E\subseteq \SX\times\SF\times\SF$.  Define $T_\fcn$ to contain the following:
		\begin{itemize}
		\item $(\forall x\in\SX)$ ``$E_x = E(x,\cdot,\cdot)$ is an equivalence relation on $\SF$.''
		\item $(\forall u,u'\in\SF)\big(u=u'\bic (\forall x\in\SX)E(x,u,u') \big)$.
		\item ``$<_\SX$ is a linear order of $\SX$'' and ``$<_\SF$ is a linear order of $\SF$.''
		\item For all $n$, 
		$$(\forall x_0<_\SX\cdots<_\SX x_{n-1}<_\SX x_n\in \SX)(\forall u\in\SF)\textnormal{``$u/\cap_{i\leq n}E_{x_i}$ is an $<_\SF$-convex 	subset of $u/\cap_{i<n}E_{x_i}$''} $$
		\end{itemize}
	\item Let $\FF_\fcn$ be the smallest fragment of $\L_\fcn$ containing the formula $(\exists x\in X)\neg E(x,u,u')$.
	\item For each $0<n<\omega$, let $T_{\eq:n}$ be the the $\FF_\fcn$-universal theory obtained from $T_\fcn$ by adding the sentence
$$(\forall x_0,...,x_{n-1},x_n\in \SX)\bigvee_{i<n}x_n=x_i.$$
	\end{itemize}
	
%
%
%

\item $T_\ceq$: 
	\begin{itemize}
	\item Let $\L_\ceq = \{E, <\}$ be a language with two binary relations.  Define the theory of \emph{convexly-ordered equivalence relations}, $T_\ceq$, to be the theory in $\L_\ceq$ given by defining $E$ to be an equivalence relation and $<$ to be a linear ordering such that equivalence classes are convexly ordered, i.e. the following holds in any model of $T_\ceq$:
$$a<b<c \wedge E(a,c) \Rightarrow E(a,b) \wedge E(b,c)$$
	\item Let $\FF$ denote the trivial fragment.
	\end{itemize}
\end{itemize}
\end{defn}

To conclude this subsection, we introduce the notion of an indiscernible picture of a model of $T_0^*$ for a \Fraisse-like theory $(T_0,\FF)$. Of course, not all \Fraisse-like theories admit indiscernible pictures, but the definition is sensible even so.
By $\tp_\FF^\A(\aa)$ we mean the formulas in the fragment $\FF$ satisfied by $\aa$ in $\A$.
\begin{defn}[Indiscernible pictures]
Let $(T_0,\FF)$ be a \Fraisse-like theory, and let $\A\models T_0^*$ (possibly uncountable). Let $\M$ be an infinite structure. An {\em indiscernible picture} $I:\A\to\M$ (of $\A$) is an injection $I:A\to M$ such that for all $0<n<\omega$ and $\aa,\bb\in A^n$, 
$$\tp_\FF^\A(\aa) = \tp_\FF^\A(\bb)\,\,\implies\,\,\tp^\M(I\aa)=\tp^\M(I\bb).$$
By an indiscernible $T_0$-picture in $\M$, we mean an indiscernible picture $I$ of some $\A\models T_0^*$.
\end{defn}

\subsection{Ramsey and Modeling Properties}

In this subsection, we establish the twin properties that identify generalized indiscernibles (theories of indiscernibles) among all \Fraisse-like theories. 

The first of these -- the Finitary Ramsey Property -- is precisely the hypothesis that allows the classical proof of the existence of indiscernible sequences (via Ramsey Theorem and compactness)  to go through largely unchanged.

\begin{defn}[Finitary Ramsey Property]
Let $(T_0,\FF)$ be a \Fraisse-like theory.
\begin{itemize}
\item For $\aas\models T_0$, we say that $(T_0,\FF)$ has the $\aas$-Ramsey Property if for all $2\leq k<\omega$ and all $\bbs\models T_0$, there is a $\ccs\models T_0$ such that for any coloring $\xi:\emb(\aas,\ccs)\to k$, there is an embedding $u\in\emb(\bbs,\ccs)$ such that for all $f_0,f_1\in\emb(\aas,\bbs)$, $\xi(u\circ f_0) = \xi(u\circ f_1)$. (This condition is denoted by the Rado partition arrow: $\ccs\to(\bbs)^\aas_k$.)
\item $(T_0,\FF)$ has the Ramsey Property if it has the $\aas$-Ramsey Property for every $\aas\models T_0$.
\end{itemize}
\end{defn}

The second batch of definitions is more clearly identifiable as characterizing a ``theory of indiscernibles.'' The reader may find it helpful to recall that for any infinite structure $\M$ and any sequence $(a_i)_{i<\omega}$, there is an indiscernible sequence $(a_i')_{i<\omega}$ (in an elementary extension of $\M$) ``patterned'' on $(a_i)_{i<\omega}$. The first two definitions below (of $\ell$-embeddings and \EM-templates) formalize the idea of patterning. The Modeling Property, then, uses these notions to describe what is, naturally, a ``theory of indiscernibles.''

\begin{defn}[$\ell$-Embeddings]
Let $(T_0,\FF)$ be a \Fraisse-like theory, and let $\A,\B\models T_0^*$. By an ``$\ell$-embedding $f:\A\sto\B$,'' we mean a family $f = \big(f_X:\A_X\to \B\big)_{X\subset_\fin A}$ of embeddings, where for each $X\subset_\fin A$, $\A_X\prec_\FF\A$ is a finite model of $T_0$ containing $X$. Implicitly, there is a uniform bounding function $s:\omega\to\omega$ such that $|\A_X|\leq s(|X|)$ for all $X\subset_\fin A$.
\end{defn}

\begin{defn}[\EM-templates; patterning]
Let $(T_0,\FF)$ be a \Fraisse-like theory, and let $\A\models T_0^*$. In an arbitrary language $\L$, let $\M$ be some $\L$-structure. We define an {\em \EM-template} to be a family $\Gamma$ of maps $A^n\to M^n:\aa\mapsto \Gamma(\aa)$ ($0<n<\omega$) satisfying the following:
\begin{itemize}
\item For all $0<n<\omega$, $\aa\in A^n$, $\bb=\Gamma(\aa),$ and $\sigma\in \emph{Sym}(n)$,
$$\Gamma(a_{\sigma(0)},...,a_{\sigma(n-1)}) = (b_{\sigma(0)},...,b_{\sigma(n-1)}).$$
\item For all $0<m,n<\omega$, $\aa\in A^m$ and $\aa'\in A^n$, 
$$\tp^\M(\Gamma(\aa\concat\aa')) = \tp^\M(\Gamma(\aa)\concat\Gamma(\aa'))$$
\end{itemize}
Now, suppose $\M\preceq\N$, $\B\models T_0^*$, $I:\B\to\N$ is an indiscernible picture of $\B$ in $\N$. We say that {\em $I$ is patterned on $\Gamma$} if there is a family $\big(f^\Delta:\B\sto\A\big)_{\Delta\subset_\fin\L}$ of patterning maps such that for all $\Delta\subset_\fin\L$, $X\subset_\fin B$ and $\bb\in X^{<\omega}$,
$\tp_\Delta^\N(I\bb) = \tp_\Delta^\M\left( \Gamma(f_X^\Delta\bb))\right)$.
\end{defn}

\begin{defn}[Modeling Property]
A \Fraisse-like theory $(T_0,\FF)$ with generic model $\A\models T_0^*$ is said to have the {\em Modeling Property} if for every infinite structure $\M$ and every \EM-template $\Gamma:A^{<\omega}\to M^{<\omega}$, there are $\M\preceq\N$ and an indiscernible picture $I:\A\to\N$ patterned on $\Gamma$.
\end{defn}

Finally, we can state the main characterization theorem connecting the Ramsey Property and the Modeling Property. A very-slightly weaker form of this theorem was proved in \cite{ScowNIP}, and in  \cite{HillIndisc}, it was proved as one of a somewhat longer list of equivalent conditions.

\begin{thm}
Let $(T_0,\FF)$ be a \Fraisse-like theory. Then $(T_0,\FF)$ has the Finitary Ramsey Property if and only if it has the Modeling Property.
\end{thm}

\begin{defn}[Theory of indiscernibles]
A {\em theory of indiscernibles} is precisely a \Fraisse-like theory $(T_0,\FF)$ that has the Modeling Property (equivalently, the Finitary Ramsey Property).
\end{defn}

\subsection{Examples}

We hope that the following examples will be helpful to the reader in understanding the definitions that we have just laid down.

\begin{example}\label{example_mutuallyindiscernible}
$(T_n,\FF_n)$ is a \Fraisse-like theory, and $T_n^*$ is the theory asserting just that each $(\SX_i,<_i)$ is a dense linear order without endpoints. If $\A\models T_n^*$ is countable, then $Aut(\A)=Aut(\QQ,<)^n$, which is extremely amenable.
So by Lemma 6.7 of \cite{KPT} we can conclude that $(T_n,\FF_n)$ is a theory of indiscernibles.

Now, consider an indiscernible picture $I:\A\to\M$, where $\M$ is some infinite $\L$-structure. We note that for $i<j<n$ and $a\in \SX_i^\A$, $b\in \SX_i^\A$, $I(a)$ and $I(b)$ may be of different sorts. However, each sequence $I_j=(a_t)_{t\in\SX_i^\A}$ ($j<n$) is an indiscernible sequence in the usual sense, and moreover, the family $(I_0,...,I_{n-1})$ comprises mutually indiscernible sequences in the sense that for each $j<n$, $I_j$ is indiscernible over $\bigcup_{k<n:k\neq j}I_k$.
\end{example}

\begin{example}[$n$-multi-order]\label{Defn_nmultiorder}
$(\MLOn, \FFn)$ is a \Fraisse-like theory, and $\MLOn^*$ is the theory where each $<_i$ is a dense linear order and these orders are independent.  Then $\MLOn^*$ is $\aleph_0$-categorical and admits elimination of quantifiers.  Moreover, $\MLOn$ is a theory of indiscernibles.  For more details, see Definition 1.2 of \cite{GuinHillOP}, where this is called $\mathrm{MLO}_n^*$.
Also see Corollary 1.4 of \cite{Bodirsky} and Theorem 4 of \cite{Sokic}.
\end{example}

\begin{example}[function spaces]
$T_\fcn$ is a $\FF_\fcn$-universal theory. It is easy to see that $(T_\fcn,\L_\fcn)$ is a \Fraisse-like theory. In fact, the model companion shares an automorphism group with that of the Shelah tree $I_s = (\leftexp{\omega>}{\omega},\lhd,\wedge,\lx,\{P_n\}_n)$ from \cite{KimKimScow} where $\wedge$ is the meet in the partial tree order $\lhd$, $\lx$ is the lexicographical order on sequences, and $P_n$ is a unary predicate picking out  the $n$-th level of the tree.  Identify leaves of $\leftexp{n \geq}{\omega}$ with the one-point classes $[x] / \bigcap_{1 \leq 1 \leq n} E_{x_i}$.
\end{example}

\begin{example}
$(T_\ceq, \FF)$ is a \Fraisse-like theory.  A countable model of the model companion $T_\ceq^*$ is a collection of densely ordered convex equivalence classes, each of which is a dense linear order.  
\end{example}

\subsection{Collapse}

It probably is not surprising (given an understanding of stable theories) that for a given theory $T$, not every indiscernible picture of a model of $T_0^*$ ``survives'' -- sometimes, the indiscernible picture is already indiscernible relative to a reduct of $(T_0,\FF)$. This phenomenon is the basis for our approach to dividing lines in this article, so obviously, we must formalize exactly what we mean.

\begin{defn}[Reducts of \Fraisse-like theories]
Let $(T_0,\FF_0)$ and $(T_0^+,\FF_0^+)$ be \Fraisse-like theories in $\L_0$ and $\L_0^+$, respectively, and let $\A_0\models T_0^*$ and $\A_0^+\models (T_0^+)^*$ be their respective generic models.
\begin{itemize}
\item As usual, we say that $\A_0$ is a reduct of $\A_0^+$, if there is a bijection $w:A_0\to A_0^+$ such that for every formula $\phi(x_0,...,x_{n-1})$ of $\L_0$, there is a formula $\psi(x_0,...,x_{n-1})$ of $\L_0^+$ such that 
$$\A_0\models \phi(\aa)\iff \A_0^+\models\psi(w\aa)$$
for all $\aa\in A_0^n$. As all of the theories in question are $\aleph_0$-categorical, this is equivalent to asserting that, up to the identification provided by $w$, every complete type of $T_0^*$ over $\emptyset$ is equal to a (finite) union of complete types of $(T_0^+)^*$ over $\emptyset$.
We will say that $(T_0,\FF_0)$ is a reduct of $(T_0^+,\FF_0^+)$ if the generic model of the former is a reduct of the generic model of the latter.
\end{itemize}
\end{defn}

\begin{defn}[Reducts of \Fraisse-like theories]
Let $(T_0,\FF_0)$ and $(T_0^+,\FF_0^+)$ be theories of indiscernibles, and let $T$ be an arbitrary complete theory. Assume $(T_0,\FF_0)$ is a reduct of $(T_0^+,\FF_0^+)$ -- so fix $w:A_0\to A_0^+$ witnessing this, where $\A_0$ and $\A_0^+$ are the generic models of the two theories, respectively. 
\begin{itemize} 
\item  Let $I:\A_0^+\to \M$ be an indiscernible picture of  $\A_0^+\models (T_0^+)^*$ in a model $\M$ of $T$. We say that {\em $I$ collapses to $(T_0,\FF_0)$} if the composite $I\circ w:\A_0\to\M$ is an indiscernible picture of $\A_0$ in $\M$.
\item We say that {\em $T$ collapses $(T_0^+,\FF_0^+)$ to $(T_0,\FF_0)$} if every indiscernible picture $I:\A_0^+\to \M$ of $\A_0^+\models (T_0^+)^*$ in a model $\M$ of $T$ collapse to $(T_0,\FF_0)$.

\end{itemize}
\end{defn}

\begin{rem}
We note that our formalization of collapse requires that a specific reduct is isolated before hand. Although it is also possible (presumably) to speak of ``collapse'' without requiring that the remainder is anything in particular, this is {\em not} the approach we adopt here.
\end{rem}

\subsection{Dividing across}

Although we can get similar collapse results for op-dimension and rosiness, for NTP2, we will need to consider a slightly different notion.  This is primarily because of the fact that, while stability, NIP, op-dimension, and rosiness are characterized by the non-existence of a pattern with only existential criteria, NTP2 (and likewise, NTP1 for that matter) is characterized by the non-existence of a pattern with both existential and universal requirements.  So we will need a notion of dividing across a fixed sequence $I$ to deal with the ``universal part'' of the pattern.

\begin{defn}[Dividing across $I$]\label{Defn_DivideAcross}
Let $I:\A\to\M^y$ be an indiscernible picture of $\A\models T_0^*$, and let $\phi(x;y)$ be an $\L$-formula. We say that $\phi(x;y)$ {\em divides across $I$} if there is an indiscernible sequence $\langle a_i \rangle_{i < \omega}$ in $\A$ such that:
\[
 \left\{\phi(x;Ia_i)\right\}_{i < \omega} \text{ is inconsistent.}
\]
\end{defn}

Notice that saying $\phi(x;y)$ divides across $I$, witnessed by $\langle a_i \rangle_{i < \omega}$, shows that $\varphi(x; Ia_0)$ divides over $\emptyset$, since $\langle Ia_i \rangle_{i < \omega}$ is an indiscernible sequence in $\M$.  As usual, by compactness, this is equivalent to demanding that $\left\{\phi(x;Ia_i)\right\}_{i < \omega}$ is $k$-inconsistent for some $k < \omega$.

We will use this notion in Section \ref{Sect_NTP2} below.  We will introduce a specialized versions of dividing, namely ``dividing vertically.''
To connect diving to indiscernible collapse,
we will also have a notion of \emph{coding} formulas.  Though this definition depends on the specific index theory $T_0$, the basic idea is to have a single $\L$-formula witnessing the collapse of a particular relation symbol in $\L_0$ for a indiscernible picture $I$.

\section{$op$-Dimension and indiscernible multi-orders}\label{Sect_opDim}


Let $T$ be a complete first-order $\L$-theory and let $\UU$ be a monster model for $T$.

\subsection{Background on op-dimension}

There are many equivalent definitions for op-dimension given in \cite{GuinHillOP}.  The main one of interest here will be the following.

\begin{defn}[op-dimension via IRD-patterns]
 We say that a partial type $\pi(x)$ has \emph{op-dimension} $\ge n$ if there exists formulas $\psi_i(x;y_i)$ and sequences $\langle b_{i,j} : j < \omega \rangle$ for $i < n$ such that, for all functions $f : n \rightarrow \omega$, the following partial type is consistent
 \[
  \pi(x) \cup \{ \psi_i(x; b_{i,j}) : j < f(i), i < n \} \cup \{ \neg \psi_i(x; b_{i,j}) : j \ge f(i), i < n \}.
 \]
 We will call such $\langle \psi_i, \langle b_{i,j} : j < \omega \rangle : i < n \rangle$ an \emph{IRD-pattern} of length $n$ and depth $\omega$ in $\pi(x)$.  We will also denote this by $\opD(\pi) \ge n$.  We will write $\opD(a/B)$ for $\opD(\tp(a/B))$.  If the op-dimension of $\pi$ is $\ge n$ for all $n$, we say it is infinite and write $\opD(\pi) = \infty$.
\end{defn}

Notice that by the modeling property in Example \ref{example_mutuallyindiscernible}, we may assume that the $b_{i,j}$ in the definition of op-dimension are such that $\{ \langle b_{i,j} : j \in \mathbb{Q} \rangle : i < n \}$ is mutually indiscernible (and indexed by $\mathbb{Q}$ or, in fact, any infinite order $J$).

Notice that if $x=x$ has finite op-dimension, then in particular $T$ has NIP in the sort of $x$.  The converse does not hold.  Consider $T_{\omega-\mathrm{mlo}}^*$ (see Definition \ref{Defn_index}).  A partial type $\pi$ is stable if and only if $\opD(\pi) = 0$.  It is easy to check that the op-dimension of a partial type $\pi$ is bounded above by the dp-rank of $\pi$ (any IRD-pattern can be converted to an ICT-pattern).  In general, however, op-dimension can be much smaller (e.g., there are stable, non-strongly dependent theories).  On the other hand, op-dimension enjoys many of the same properties as dp-rank, including subadditivity (see Theorem 2.2 of \cite{GuinHillOP}).  That is,
\[
 \opD(ab/A) \le \opD(a/A) + \opD(b/Aa).
\]
If $T$ is o-minimal (or simply distal), then dp-rank and op-dimension coincide (Theorem 3.1 of \cite{GuinHillOP}).

\begin{thm}[Theorem 1.21 and Proposition 2.1 of \cite{GuinHillOP}]
 For a partial type $\pi(x)$ over a set $B$, the following are equivalent
 \begin{enumerate}
  \item $\opD(\pi) \le n$,
	\item For all formulas $\varphi(x;y)$, all sequences $\langle b_i : i \in \mathbb{Q} \rangle$ indiscernible over $B$, and for all $a \models \pi$, there exists a partition of $\mathbb{Q}$ into convex sets $C_0 < ... < C_n$ such that, for all $\ell \le n$, the set $\{ i \in C_\ell : \models \varphi(a; b_i) \}$ is either finite or cofinite in $C_\ell$.
	\item For all $\{ \langle b_{i,j} : j \in J \rangle : i \le n \}$ almost mutually indiscernible over $B$ and all $a \models \pi$, there exists $i \le n$ such that $\langle b_{i,j} : j \in J \rangle$ is almost indiscernible over $Ba$.
 \end{enumerate}
\end{thm}

\subsection{$n$-Multi-order characterization of op-dimension}

\begin{prop}[Proposition 1.18 of \cite{GuinHillOP}]\label{Prop_118GuinHillOP}
 For a partial type $\pi(x)$ and $n < \omega$, the following are equivalent:
 \begin{enumerate}
  \item $\opD(\pi) \ge n$,
	\item There exists an indiscernible $\MLOn$-picture $I: \A \to \pi(\UU)$ and a formula $\varphi(x;y)$ such that, for all $i < n$ and for all $X \subseteq A$ a cut of $(A; <_i)$, there exists $b \in \UU$ such that $X = \{ a \in \A : \models \varphi(Ia; b) \}$.
 \end{enumerate}
\end{prop}

Along similar lines, we can use collapse of indiscernible $\MLOn$-pictures to characterize op-dimension, supposing the ambient theory has NIP.  The following is the main result of this section.

\begin{thm}[Characterization of op-dimension by collapsing indiscernibles]\label{Thm_opDimIndisc}
 Suppose $T$ has NIP.  For a partial type $\pi(x)$ and $n < \omega$, the following are equivalent:
 \begin{enumerate}
  \item $\opD(\pi) < n$,
	\item For all indiscernible $\MLOn$-pictures $I: \A \to \pi(\UU)$, there exists $j < n$ such that $I$ is an indiscernible picture for $\FFn^{-j} := \{ <_i : i \neq j, i < n \}$
 \end{enumerate}
\end{thm}



\begin{proof} 
 (1) $\Rightarrow$ (2): Suppose (2) fails.  That is, there exists an indiscernible $\MLOn$-picture $I: \A \to \pi(\UU)$ such that, for all $j < n$, $I$ is not an indiscernible picture for $\FFn^{-j} := \{ <_i : i \neq j, i < n \}$.  That is, there exists $k < \omega$ and $\aa^j, \bb^j \in A^k$ such that
 \[
  \tp_{\FFn^{-j}}^\A(\aa^j) = \tp_{\FFn^{-j}}^\A(\bb^j), \tp_{\FFn}^\A(\aa^j) \neq \tp_{\FFn}^\A(\bb^j), \text{ and } \tp_{\L}^{\UU}(I\aa^j) \neq \tp_\L^\UU(I\bb^j).
 \]
 We may assume the same $k$ works for all $j < n$ (take the maximal such) and the coordinates of $\aa^j$ are distinct and the coordinates of $\bb^j$ are distinct.  Without loss of generality, by reordering, suppose
 \[
  a^j_0 <_j a^j_1 <_j ... <_j a^j_{k-1}.
 \]
 Since transpositions generate the symmetric group on $k$, we may assume that the $<_j$-order type of $\aa^j$ and $\bb^j$ differ by transposing two elements.  That is, for some $t_j < k-1$,
 \[
  b^j_0 <_j ... <_j b^j_{t_j-1} <_j b^j_{t_j+1} <_j b^j_{t_j} <_j b^j_{t_j+2} <_j ... <_j b^j_{k-1}.
 \]
 Since $\tp_\L^\UU(I\aa^j) \neq \tp_\L^\UU(I\bb^j)$, there exists $\psi_j(y_0, ..., y_{k-1}) \in \L$ such that
 \[
  \models \psi_j(I\aa^j) \wedge \neg \psi_j(I\bb^j).
 \]
 Fix $m < \omega$.  We use this as a template to construct an IRD-pattern of depth $m$ and length $n$ in $\pi$, showing that $\opD(\pi) \ge n$ (and hence (1) fails).

 For each $j < n$, choose
 \[
  d^j_0, ..., d^j_{t_j-1}, d^j_{t_j+1}, ..., d^j_{k-1} \in {}^n \mathbb{Q}
 \]
 such that, for all $i < n$,
 \begin{itemize}
  \item $d^j_\ell(i) < d^j_s(i)$ if $a^j_\ell <_i a^j_s$ for all $\ell, s \neq t_j$,
	\item $d^j_\ell(i) < -1$ if $a^j_\ell <_i a^j_{t_j}$ for all $\ell \neq t_j$,
	\item $d^j_\ell(i) > m$ if $a^j_\ell >_i a^j_{t_j}$ for all $\ell \neq t_j$.
 \end{itemize}
 Thereby, for all $f \in {}^n (-1,m)_\mathbb{Q}$, for all $i$,
 \[
  \langle d^j_0, ..., d^j_{t_j-1}, f, d^j_{t_j+1}, ..., d^j_{k-1} \rangle
 \]
 has the same $<_i$-order type as $\aa$.  In other words, for all $f \in {}^n (-1,m)_\mathbb{Q}$,
 \[
  \tp_{\FFn}^{{}^n \mathbb{Q}}(d^j_0, ..., d^j_{t_j-1}, f, d^j_{t_j+1}, ..., d^j_{k-1}) = \tp_{\FFn}^\A(\aa^j).
 \]
 For each $j < n$ and $\ell < m$, define $e^j_\ell \in {}^n \mathbb{Q}$ as follows: For each $i < n$,
 \begin{itemize}
  \item $e^j_\ell(i) = d^j_{t_j+1}(i)$ for all $i \neq j$,
	\item $e^j_\ell(j) = \ell - 1/2$.
 \end{itemize}
 So, for $e^j_\ell$, we are merely taking $d^j_{t_j+1}$ and stretching it, in the $j$th coordinate, across $(-1,m+1)_\mathbb{Q}$.  Now, for any $f \in {}^n m$, for any $j < n$ and any $\ell < m$, we have
 \begin{itemize}
  \item $d^j_0 <_j ... <_j d^j_{t_j-1} <_j f <_j e^j_\ell <_j d^j_{t_j+2} <_j ... <_j d^j_{k-1}$ for $f(j) < \ell$ and,
	\item $d^j_0 <_j ... <_j d^j_{t_j-1} <_j e^j_\ell <_j f <_j d^j_{t_j+2} <_j ... <_j d^j_{k-1}$ for $f(j) \ge \ell$.
 \end{itemize}
 In other words, for any $f \in {}^n m$, $\ell < m$, and $j < n$,
 \begin{itemize}
  \item $\tp_{\FFn}^{{}^n \mathbb{Q}}(d^j_0, ..., d^j_{t_j-1}, f, e^j_\ell, d^j_{t_j+2}, ..., d^j_{k-1}) = \tp_{\FFn}^\A(\aa^j)$ for $f(j) < \ell$ and,
	\item $\tp_{\FFn}^{{}^n \mathbb{Q}}(d^j_0, ..., d^j_{t_j-1}, f, e^j_\ell, d^j_{t_j+2}, ..., d^j_{k-1}) = \tp_{\FFn}^\A(\bb^j)$ for $f(j) \ge \ell$.
 \end{itemize}
 After embedding into $\A$ (by first extending each $<_i$ arbitrarily to be a linear order), we see that, for all $f \in {}^n m$, $\ell < m$, and $j < n$,
 \[
  \models \psi_j(Id^j_0, ..., Id^j_{t_j-1}, If, Ie^j_\ell, Id^j_{t_j+2}, ..., Id^j_{k-1}) \text{ iff. } f(j) < \ell.
 \]
 Let $\psi'_j(y; y_0, ..., y_{t_j-1}, y_{t_j+1}, ..., y_{k-1}) = \psi_j(y_0, ..., y_{t_j-1}, y, y_{t_j+1}, ..., y_{k-1})$, let $\alpha_f = If \models \pi$ for all $f \in {}^n m$, and let
 \[
  \beta_{j,\ell} = \langle Id^j_0, ..., Id^j_{t_j-1}, Ie^j_\ell, Id^j_{t_j+2}, ..., Id^j_{k-1} \rangle
 \]
 for all $j < n$ and $\ell < m$.  Then, for all $j < n$, $\ell < m$, and $f \in {}^n m$,
 \[
  \models \psi'_j(\alpha_c; \beta_{j,\ell}) \text{ iff. } f(j) < \ell.
 \]
 Hence, $\psi'_j$ and $\beta_{j,\ell}$ form an IRD-pattern of depth $m$ and length $n$ in $\pi$.  Since $m$ was arbitrary, by compactness, we get that $\opD(\pi) \ge n$, as desired.


 (2) $\Rightarrow$ (1): Suppose (1) fails, so $\opD(\pi) \ge n$.  By Proposition \ref{Prop_118GuinHillOP}, there exists an indiscernible $\MLOn$-picture $I: \A \to \pi(\UU)$ and a formula $\varphi(x;y)$ such that, for all $j < n$ and for all $X \subseteq A$ a cut of $(A; <_j)$, there exists $b \in \UU$ such that $X = \{ a \in \A : \models \varphi(Ia; b) \}$.  Since $T$ has NIP, $\varphi(x;y)$ has NIP, so fix $m < \omega$ such that $\varphi$ has VC-dimension $< m$.  That is, 
 \[
  \models \neg \exists x_0, ..., x_{m-1} \bigwedge_{s \in {}^m 2} \left( \exists y \bigwedge_{\ell < m} \varphi(x_\ell, y)^{s(\ell)} \right).
 \]

 We show that (2) fails on $I$.  Fix $j < n$ and suppose, by means of contradiction, that $I$ is an indiscernible picture for $\FFn^{-j} := \{ <_i : i \neq j, i < n \}$.  Fix any $\cc \in \A^m$.  By choice of $m$ (i.e., that $m$ is greater than the VC-dimension of $\varphi$), there exists $s \in {}^m 2$ such that
\[
 \models \neg \exists y \bigwedge_{\ell < m} \varphi(Ic_\ell, y)^{s(\ell)}.
\]
Let $L_0 = \{ \ell < m : s(\ell) = 0 \}$ and $L_1 = \{ \ell < m : s(\ell) = 1 \}$.  Choose $\dd \in \A^m$ such that
 \begin{itemize}
  \item $\tp_{\FFn^{-j}}^{\A}(\cc) = \tp_{\FFn^{-j}}^{\A}(\dd)$, and
	\item For all $\ell_1 \in L_1$ and $\ell_0 \in L_0$, $d_{\ell_1} <_j d_{\ell_0}$.
 \end{itemize}
 By indiscernibility with respect to $\FFn^{-j}$,
 \[
  \models \neg \exists y \bigwedge_{\ell < m} \varphi(Id_\ell, y)^{s(\ell)}.
 \]
 Without loss of generality, by reordering, suppose $d_0 <_j d_1 <_j ... <_j d_{m-1}$, so we have
 \[
  \models \neg \exists y \left( \bigwedge_{\ell < m_0} \varphi(Id_\ell; y) \wedge \bigwedge_{m_0 \le \ell < m} \neg \varphi(Id_\ell; y) \right),
 \]
 where $m_0 = |L_1|$.  Now choose any $X \subseteq A$ a cut of $(A; <_j)$ with $d_\ell \in X$ if and only if $\ell < m_0$.  Therefore, there exists $b \in \UU$ such that $X = \{ a \in \A : \varphi(Ia; b) \}$.  In particular,
 \[
  \models \bigwedge_{\ell < m_0} \varphi(Id_\ell, b) \wedge \bigwedge_{m_0 \le \ell < m} \neg \varphi(Id_\ell, b).
 \]
 This is a contradiction.  Hence (2) fails and we are done.

\end{proof}

Note that we only use NIP in the proof of Theorem \ref{Thm_opDimIndisc}, (2) $\Rightarrow$ (1).  However, this use is necessary.  Consider the theory $T$ of the random graph (where $\L = \{ R \}$) and take $\pi(x) = [x=x]$ in the homesort.  Then, any indiscernible sequence is an indiscernible set (either everything is $R$-connected or $R$-disconnected).  That is, Theorem \ref{Thm_opDimIndisc} (2) holds for $n = 1$.  On the other hand, Theorem \ref{Thm_opDimIndisc} (1) fails for all $n$ (i.e., $\opD(\pi) = \infty$).  To see this, notice that $R(x;y)$ together with any distinct $\langle \langle b_{i,j} : j < \omega \rangle : i < n \rangle$ form an IRD-pattern of length $n$.  Indeed, for any $f : n \rightarrow \omega$, the type
\[
 \{ R(x; b_{i,j}) : i < n, j < f(i) \} \cup \{ \neg R(x; b_{i,j}) : i < n, f(i) \le j < \omega \}
\]
is consistent by saturation.

If we assume that $T$ is o-minimal, since dp-rank and op-dimension coincide, then this collapse of multi-order indiscernibles characterizes dp-rank.


\section{Rosiness and indsicernible function spaces}\label{Sect_Rosiness}

\subsection{Background on rosiness}

There are several equivalent ways to define rosy theories, but for the sake of brevity, we present the one that is closest to the work we will carry in this article. This definition is based on the following notion of ``equivalence relation ranks'' as proposed in \cite{EalyOnshuus}.

We remark that when defined in terms of local \th-ranks or the behavior of \th-independence, it is very important to work with a theory that eliminates imaginaries; that is, the class of rosy theories is strictly contained the class of ``real-rosy theories.'' However, the characterization in terms of equivalence relation ranks eliminates this discrepancy -- a theory $T$ has finite equivalence relation ranks if and only if $T^\eq$ has finite equivalence relation ranks.

\begin{defn}[Equivalence-relation-rank: $\eqrk$]
Let $\pi(x)$ be a partial type, and let $\Delta$ be a finite set of formulas $\delta(x,x',z)$ such that for every $c$, $\delta(x,x',c)$ defines an equivalence relation. (We call such a $\Delta$ a {\em finite set of families of equivalence relations}.)
\begin{itemize}
\item $\eqrk(\pi(x),\Delta)\geq0$
\item For a limit ordinal $\lambda$, $\eqrk(\pi(x),\Delta)\geq\lambda$ if $\eqrk(\pi(x),\Delta)\geq\alpha$ for all $\alpha<\lambda$.
\item $\eqrk(\pi(x),\Delta)\geq\alpha+1$ if there are a formula $\delta(x,x',z)\in \Delta$ and a parameter $c$ such that, with $E(x,x') = \delta(x,x',c)$, the set,
$$\left\{b/E:b\models\pi\wedge\eqrk(\pi(x)\cup\{E(x,b)\},\Delta)\geq\alpha\right\}$$
is infinite.
\item $\eqrk(\pi(x),\Delta)=\infty$ if $\eqrk(\pi(x),\Delta)\geq\alpha$ for every ordinal $\alpha$.
\end{itemize}
\end{defn}

\begin{rem}\label{fact:omega-is-infty}
Let $\pi(x)$ be a partial type, and let $\Delta$ be a finite set of families of equivalence relations. By compactness, if $\eqrk(\pi(x),\Delta)\geq\omega$, then $\eqrk(\pi(x),\Delta)=\infty$.
\end{rem}

The key theorem of \cite{EalyOnshuus} (for our purposes) is the following.

\begin{thm}[\cite{EalyOnshuus}]
Let $T$ be a theory with big model $\UU$. $T$ is rosy if and only if $\eqrk(\pi(x),\Delta)<\omega$ for every partial type $\pi(x)$ and every finite set $\Delta$ of families of equivalence relations.
\end{thm}

\medskip

In the remainder of this subsection, we define a variant of the equivalence relation ranks as a coding property, and we show that non-rosiness can also be characterized by ``coding equivalence relations.''

\begin{defn}
Let $T$ be a theory with big model $\UU$. For a formula $\phi(x;y_0,y_1)$, we say that $\phi$ {\em codes equivalence relations} if for every $0<n<\omega$, there are $(a_{i})_{i<n}$ of sort $x$ and $(b_f)_{f\in {^n}n}$ of sort $y$ such that for all $i<n$ and $f_0,f_1\in {^n}n$, 
$$\models\phi(a_{i};b_{f_0},b_{f_1})\,\,\iff\,\,f_0(i)=f_1(i).$$
\end{defn}

\begin{lemma}\label{lemma:non-rosy-equals-coding}
Let $T$ be a theory with big model $\UU$, and assume $T$ eliminates imaginaries. $T$ is non-rosy if and only if there is a  partitioned formula $\phi(x;y_0,y_1)$ that codes equivalence relations in $\UU$.
\end{lemma}
\begin{proof}
First, suppose $\phi(x;y_0,y_1)$ codes equivalence relations. Then, let 
$$\delta(y_0,y_1;x)=\forall y\left(\phi(x;y_0,y)\bic\phi(x;y_1,y)\right).$$
By compactness and the fact that $\phi(x;y_0,y_1)$ codes equivalence relations, there are $(a_{i})_{i<\omega}$ and $(b_f)_{f\in {^\omega}\omega}$ such that for all $i<\omega$ and $f_0,f_1\in {^\omega}\omega$, 
$$\models\phi(a_{i};b_{f_0},b_{f_1})\,\,\iff\,\,f_0(i)=f_1(i).$$
Now, we make a simple observation:
Let $\sigma\in{^{<\omega}}\omega$, and for each $j<\omega$, let $\sigma\concat j\subset f_j\in{^\omega}\omega$; then, 
$$\eqrk\left(\bigwedge_{\ell<|\sigma|}\delta(y,b_{f_0};a_\ell)\wedge\delta(y,b_{f_j};a_{|\sigma|}),\left\{\delta\right\}\right)\geq0$$
for every $j<\omega$. From this, it follows immediately that $\eqrk(x{=}x,\{\delta\})\geq\omega$, and this shows that $T$ is not rosy.

\medskip
Conversely, suppose $T$ is not rosy. Through simple coding tricks, we may assume that there is a single family of equivalence relations $\delta(x,x';z)$ such that  $\eqrk(x{=}x,\{\delta\})=\infty$. 
From this, we recover $(c_\tau)_{\tau\in {^{<\omega}}\omega}$ and $(d_\sigma)_{\sigma\in {^{<\omega}}\omega}$ such that, for any $\sigma\in {^{<\omega}}\omega$ and $\tau = \sigma\concat j$ ($j<\omega$)  
$$\eqrk\left(\bigwedge_{\ell<|\sigma|}\delta(x,c_{\tau\r(\ell+1)};d_{\sigma\r \ell}),\{\delta\}\right)\geq\omega.$$
By the Pigeonhole Principle, we may assume that $c_{\sigma\concat j}d_\sigma\equiv c_{\sigma'\concat j'}d_{\sigma'}$ whenever $|\sigma|=|\sigma'|$. Now, let $0<n<\omega$ be given. Applying several automorphisms, we may assume that $d_\sigma=d_{\sigma'}=:a_{|\sigma|}$ whenever $\sigma,\sigma'\in {^{{<}n}}n$ satisfy $|\sigma|=|\sigma'|$; then, for $f\in {^n}n$, we choose any $b_f\models\bigwedge_{\ell<n}\delta(x,c_{f\r(\ell{+}1)};d_{f\r\ell})$. As $n$ was arbitrary, we have shown that $\phi(x;y_0,y_1)=\delta(y_0,y_1;x)$ codes equivalence relations.
\end{proof}


\subsection{Ramsey theory for function spaces}

We have established that non-rosiness is equivalent to ``coding equivalence relations.'' In this subsection, develop of theory of indiscernibles that corresponds to the latter notion.

\begin{prop}
$(T_\fcn,\L_\fcn)$ is a theory of indiscernibles.
\end{prop}
\begin{proof}

First we prove that $(T_{\eq:n},\FF_\fcn)$ is a theory of indiscernibles.
The base case follows since $(T_{\eq:1},\FF_\fcn)$ is interdefinable with the theory of linear order.


Now suppose the generic model of $T_{\eq:n}$ has extremely amenable automorphism group $A$, and consider $T_{\eq:(n+1)}$.  Call the $E_{x_i}$ simply $E_i$ for $0 \leq i \leq n$.  The $X$ sort is finite and rigid. Thus the automorphism group acts on the $\bigcap_{i\leq n-1} \leq  E_i$ classes (and permutes any point(s) in the classes.)  Adding a new equivalence relation cross-cuts each $\bigcap_{i\leq n-1} \leq  E_i$ class.  Thus the automorphism group of the generic model of $T_{eq:(n+1)}$ is $\textrm{Aut}(\mathbb{Q},<)^\mathbb{Q} \rtimes A $, which we know to be extremely amenable.

To complete the argument for $(T_\fcn,\FF_\fcn)$, there are two ways to proceed. Firstly, one can observe that the automorphism groups of the generic models of the $(T_{\eq:n},\FF_\fcn)$'s form a directed family of subgroups of the automorphism group of the generic model of $(T_\fcn,\FF_\fcn)$, and that union of this family is dense the latter; by Lemma 6.7 of \cite{KPT}, it follows that $(T_\fcn,\FF_\fcn)$ is a theory of indiscernibles. A second approach would be to use compactness and the Modeling Properties of the $(T_{\eq:n},\FF_\fcn)$'s to prove that $(T_\fcn,\FF_\fcn)$ also has the Modeling Property.
\end{proof}

Now that we know what theory of indiscernibles to work with, we can state the main characterization  theorem of this section. Its proof comprises all of the next subsection.
\begin{thm}\label{thm:collapse-result-rosiness}
Let $T$ be a theory with big model $\UU$, and assume $T$ eliminates imaginaries. $T$ is rosy if and only if every indiscernible $T_\fcn$-picture in $\UU$ is  an indiscernible $T_2$-picture via $\SX_0 = \SX$, $\SX_1 = \SF$.
\end{thm}


\subsection{Proof of Theorem \ref{thm:collapse-result-rosiness}}

The proof of Theorem \ref{thm:collapse-result-rosiness} amounts to the conjunction of Propositions \ref{prop:char-rosy-necessity} and \ref{prop:char-rosy-sufficiency}, which we prove in turn.
Throughout this subsection, assume that $T$ eliminates imaginaries.

\begin{prop}\label{prop:char-rosy-necessity}
Let $T$ be a theory with big model $\UU$. If $T$ is rosy, then every indiscernible $T_\fcn$-picture in $\UU$ is  an indiscernible $T_2$-picture via $\SX_0 = \SX$, $\SX_1 = \SF$.
\end{prop}

Before proving this proposition, we will need to establish the following lemma.

\begin{lemma}\label{lemma:indisc-to-fs}
Let $T$ be a theory with big model $\UU$. If there is an indiscernible $T_\fcn$-picture in $\UU$ that is not an indiscernible $T_2$-picture via $\SX_0 = \SX$, $\SX_1 = \SF$, then there is a formula $\phi(x;y,y')$ that codes equivalence relations in $\UU$.
\end{lemma}
\begin{proof}
Let $I:\A\to \UU$ be an indiscernible picture of the countable model $\A\models T_\fcn^*$, and suppose $I$ is not an indiscernible $T_2$-picture. In particular, there are $0<n_\SX,n_\SF<\omega$, $a_0<\cdots<a_{n_\SX-1}$ and $a'_0<\cdots<a'_{n_\SX-1}$ in $\SX^\A$, $b_0<\cdots<b_{n_\SF-1}$ and $b'_0<\cdots<b'_{n_\SF-1}$ in $\SF^\A$, and a formula $\phi(x;y)$ in the language of $T$ such that:
\begin{itemize}
\item $x=(x_0,...,x_{n_\SX-1})$ and $y = (y_0,...,y_{n_\SF-1})$
\item $\qtp^\A(a,b)\neq\qtp^\A(a',b')$
\item $\UU\models\phi(Ia,Ib)$ and $\UU\models\neg \phi(Ia',Ib')$
\end{itemize}
By model-completeness and the indiscernibility of $I$, we may assume that $a'=a$. We will also take the lemma as obvious in the case that $n_\SF=1$ (using $\phi(x,y)\wedge\phi(x,y')$ to code equivalence relations), so in all that follows, we assume $n_\SF\geq 2$.

\begin{claim}
We may assume that $p_b(a,y) = \qtp^\A(a,b)$ and $p_{b'}(a,y)=\qtp^\A(a,b')$ differ only in the following way: 
There are $i_0<n_\SX$ and $t_0<n_\SF$ such that for all $i<n_\SX$, $s,t<n_\SF$,
\begin{align*}
i\neq i_0\,\,&\implies\,\,\A\models (E(a_i,b_s,b_t)\bic E(a_i,b'_s,b'_t))\\
t_0\notin\{s,t\}\,\,&\implies\,\,\A\models (E(a_i,b_s,b_t)\bic E(a_i,b'_s,b'_t)),\\
|b_{t_0}/E_{a_{i_0}}|\geq 2\,\,&\wedge\,\, |b'_{t_0}/E_{a_{i_0}}|=1.
\end{align*}
\end{claim}
\begin{proof}[Proof of claim]
Let $G$ be the set of all functions $g:n_\SX{\times} n_\SF{\times} n_\SF\to 2$ satisfying the following conditions:
\begin{enumerate}
\item For all $i<n_\SX$ and $t_0,t_1,t_2<n_\SF$, 
$$t_0=t_1\implies g(i,t_0,t_1) = 1,\,\,g(i,t_0,t_1) = g(i,t_1,t_0),$$
$$g(i,t_0,t_1) = 1\wedge g(i,t_1,t_2)=1\implies g(i,t_0,t_2) = 1.$$
\item For all $k< n_\SX$ and $i_0<\cdots<i_{k-1}<i_k<n_\SX$, for all $t_0<n_\SF$,
$$\left\{t<n_\SX:\bigwedge_{j\leq k}g(i_j,t,t_0)=1\right\}$$
is a convex subset of 
$$\left\{t<n_\SX:\bigwedge_{j< k}g(i_j,t,t_0)=1\right\}.$$
\end{enumerate}
Then $G$ is precisely the set of functions $n_\SX{\times} n_\SF{\times} n_\SF\to 2$ that correspond naturally to quantifier-free-complete types $p(a,y)$ extending $\qtp^\A(a)$; we write $p\mapsto g_p$ and $g\mapsto p_g$ for the two directions of the bijection. Let's also name the function $h:n_\SX{\times} n_\SF{\times} n_\SF\to 2$ given by, 
$$h(i,t_0,t_1) = \begin{cases}
1 & \textnormal{ if $t_0=t_1$}\\
0 &\textnormal{ otherwise.}
\end{cases}$$It's clear that $h\in G$. For any $g\in G$, we define $g^*:n_\SX{\times} n_\SF{\times} n_\SF\to 2$ as follows:
\begin{itemize}
\item For $t<n_\SF$, we say that $t$ is \emph{$g$-isolated} if for all $i<n_\SX$, $g(i,t,t') = 1\implies t=t'$.
\item Set
$$\begin{aligned}
t^* &= \max\left\{t:\textnormal{$t$ is not $g$-isolated}\right\},\\
i^* &= \max\left\{i:(\exists t)\,t\neq t^*\wedge g(i,t,t^*)=1\right\}.
\end{aligned}$$
\item Finally, define $g^*$ by
$$g^*(i,t_0,t_1) = \begin{cases}
g(i,t_0,t_1) &\textnormal{ if $i\neq i^*$ or $t^*\notin\{t_0,t_1\}$}\\
1 & \textnormal{ if $i=i^*$ and $t_0=t_1=t^*$}\\
0 &\textnormal{ if $i=i^*$ and $\{t^*\}\subsetneq\{t_0,t_1\}$.}
\end{cases}$$
\end{itemize}
It is not hard to verify that $g^*$ remains in $G$, and it is equally clear that for each $g\in G$, there is a number $m<\omega$ such that if $g_0 = g$, and $g_{k+1}=g_k^*$ for each $k<m$, then $g_m = h$. Consequently, if $g_{p_b}$ and $g_{p_{b'}}$ are the members of $G$ corresponding to $p_b(a,y)$ and $p_{b'}(a,y)$, then there is a  ``path'' $g_0=g_{p_b},g_1,...,g_{N-1},g_N=g_{p_{b'}}$ such that for each $i<N$, either $g_{i+1} = g_i^*$ or $g_i = g_{i+1}^*$. Since $\UU\models \phi(Ia,Ib)$ and $\UU\models\neg\phi(Ia,Ib')$, by indiscernibility, there is an $i<N$ such that 
$$c\models p_{g_i}(a,y)\implies \UU\models\phi(Ia,Ic),\,\,c\models p_{g_{i+1}}(a,y) \implies \UU\models\neg\phi(Ia,Ic).$$
Then, replacing $p_b$ with $p_{g_i}$ and $p_{b'}$ with $p_{g_{i+1}}$ is sufficient for the claim.
\end{proof}

\begin{obs}
By indiscernibility and model-completeness, we may further assume that $b_j=b_j'$ whenever $j\in n_\SF\setminus\{t_0\}$.
\end{obs}
Finally, taking $s_0\in W\setminus \{t_0\}$ and
$\phi^\circ(x_{i_0};y_{s_0},y_{t_0}) =$ $$\phi(a_{0,...,i_0-1},x_{i_0},a_{i_0+1,...,n_\SX-1},b_{0,...,s_0-1},y_{s_0},b_{s_0+1,...,t_0-1},y_{t_0},b_{t_0+1,...,n_\SF-1})$$
(when $s_0<t_0$) it is not difficult to verify (following \cite{ScowNIP}) that $\phi^\circ(x;y,y')$ codes equivalence relations. (When $s_0>t_0$, we use a similar definition of $\phi^\circ$.)
\end{proof}

\begin{proof}[Proof of Proposition \ref{prop:char-rosy-necessity}]
Proving the contrapositive, if there is an indiscernible $T_\fcn$-picture in $\UU$ that is not an indiscernible $T_2$-picture via $\SX_0 = \SX$, $\SX_1 = \SF$, then by Lemma \ref{lemma:indisc-to-fs}, there is a formula $\phi(x;y,y')$ that codes function spaces in $\UU$. Then, by Lemma \ref{lemma:non-rosy-equals-coding}, the existence of a formula that codes equivalence relations in $\UU$ implies that $T$ is not rosy.
\end{proof}

\begin{prop}\label{prop:char-rosy-sufficiency}
Let $T$ be a theory with big model $\UU$.
If every indiscernible $T_\fcn$-picture in $\UU$ is  an indiscernible $T_2$-picture via $\SX_0 = \SX$, $\SX_1 = \SF$, then $T$ is rosy.
\end{prop}
\begin{proof}
For the contrapositive, suppose $T$ is not rosy. By Lemma \ref{lemma:non-rosy-equals-coding}, there is a formula $\phi(x;y,y')$ that codes function spaces in $\UU$. We will use this fact to construct an indiscernible $T_\fcn$-picture in $\UU$ that is not  an indiscernible $T_2$-picture via $\SX_0 = \SX$, $\SX_1 = \SF$.

In this paragraph, we collect the data necessary to devise an \EM-template.
By compactness, there are $(c_i)_{i<\omega}$ and $(d_f)_{f \in {^\omega}\omega}$ in $\UU$ such that 
$$\models\phi(c_i;d_{f_0},d_{f_1})\iff f_0(i)=f_1(i)$$
for all $i<\omega$ and $f_0,f_1:\omega\to \omega$.
 Let $\A\models T_\fcn^*$ be the generic model, and let $u:\SX^\A\to\omega$ and $v:\SF^\A\to{^\omega}\omega$ be injective functions such that 
$$\A\models E(a,b_0,b_1)\iff\UU\models\phi(c_{u(a)};d_{v(b_0)},d_{v(b_1)})$$
for all $a\in\SX^\A$ and $b_0,b_1\in\SF^\A$. From these, we define our \EM-template by setting
$$\Gamma(\aa\concat\aa') = (c_{u(a_i)})_{i<m}\concat(d_{v(a'_j)})_{j<n}$$
whenever $\aa\in (\SX^\A)^m$ and $\aa'\in (\SF^\A)^n$.

By the Modeling Property of $(T_\fcn,\FF_\fcn)$, we recover an indiscernible picture $I:\A\to\UU$ patterned on $\Gamma$, and we take it as obvious that $I$ is not an indiscernible $T_2$-picture via $\SX_0 = \SX$, $\SX_1 = \SF$.
\end{proof}

\section{NTP2 and Convexly-ordered Equivalence Relations}\label{Sect_NTP2}

\begin{defn}[$k$-TP2]
Let $T$ be a theory with big model $\UU$, and let $\phi(x;y)$ be a partitioned formula of the language $\L$ of $T$. For $2\leq k<\omega$, $\phi(x;y)$ is said to have the {\em $k$-tree property of the second kind} ($k$-TP2) in $T$ if there is a family $(b^i_j)_{i,j\in\omega}$ in $\UU$ such that:
\begin{itemize}
\item For every function $f:\omega\to\omega$, $\left\{\phi(x,b^i_{f(i)})\right\}_{i<\omega}$ is consistent.
\item For every $i<\omega$, $\left\{\phi(x,b^i_j)\right\}_{j<\omega}$ is $k$-inconsistent.
\end{itemize}
$T$ is said to have TP2 if there are a formula $\phi(x;y)$ and  $2\leq k<\omega$ such that $\phi$ has $k$-TP2 in $T$; otherwise, we say that $T$ is NTP2.
\end{defn}

\begin{fact}
 (Theorem 0.2 of \cite{Shelah93}) The following are equivalent:
\begin{itemize}
\item $T$ has TP2.
\item Some formula $\phi(x;y)$ has 2-TP2 in $T$.
\end{itemize}
\end{fact}

\subsection{$\ceq$-Indiscernibles}

 We will argue independently for why $T_\ceq^*$ is a theory of indiscernibles.  By \cite{KPT}, the finite models of $T_\ceq$ form a Ramsey class, and so by \cite{ScowNIP}, $T_\ceq^*$ is a theory of indiscernibles.  These indiscernibles have already been widely used in the TP2 context, typically as \emph{array indiscernibles} (as described in \cite{KimKimScow}) where the language is composed of two linear orderings.

\subsection{Characterizing NTP2}

We follow Claim 2.19 in \cite{Shelah500} in the sense that the absence of TP2 in a theory is characterized by dynamics in the EM-type of indiscernible sequences.  Namely, if the indiscernible sequence allows for a certain kind of consistency, it must then allow for more.
Since we are tracking the dynamics of formulas asserting consistency, as they show up in the EM-type, it is helpful to make use of a notation from \cite{Malliaris}:

\begin{defn}
The characteristic sequence associated to a formula $\varphi(x;y)$ is defined to be

$P_n(y_0,\ldots,y_{n-1}) \equiv \exists x (\bigwedge_{i < n} \varphi(x;y_i))$

\ms

We will write $P^{\varphi}_n(y_0, \ldots, y_{n-1})$ to denote dependence on $\varphi$.
\end{defn}

\begin{defn}[Coding equivalence]
Let $\A\models T_\ceq^*$ be a countable model, and let $I:\A\to\UU^y$ be an indiscernible picture. For a formula $\phi(x;y)\in\L(\UU)$, we say that $\phi$ {\em codes equivalence on $I$} if, for all $a_0, a_1 \in A$ with $a_0 <^\A a_1$,
\[
 E^\A(a_0, a_1) \text{ if and only if } \models \neg \phi(Ia_0; Ia_1).
\]
\end{defn}

\begin{lemma}\label{Lemma_strictpat_NTP2}
Let $T$ be a theory with big model $\UU$. Let $I:\A\to\UU$ be an indiscernible picture of $\A\models T_\ceq^*$. Then $I$ does not collapse to order if and only if there exist a formula $\phi(x;y)\in\L(\UU)$ and an embedding $f:\A\to\A$ such that $\phi$ codes equivalence on $I \circ f$.
\end{lemma}

\begin{proof}
 As before, one direction is clear, so we show the other direction.  Suppose $I$ does not collapse to order.  Then, there are $<^\A$-increasing $n$-tuples $\aa_1,\aa_2$ from $I$ such that $\qtp^\A(\aa_1) = \qtp^\A(\aa_2)$ but $I\aa_1, I\aa_2$ do not have the same complete type in $\UU$.  Thus for some formula $\theta(\xx)$ of the language of $\L$ of $T$, $\models \theta(I\aa_1)$ and $\models \neg \theta(I\aa_2)$.  Let $q_i = \qtp^\A(\aa_i)$, $i=1,2$.  In general, let $S^\ceq_n(\emptyset)$ be the set of all quantifier-free $n$-types $p(x_0, ..., x_{n-1})$ where
\begin{itemize}
 \item $p \vdash x_i < x_j$ for all $i < j < n$,
 \item $p \vdash E(x_i, x_i)$ for all $i < n$,
 \item if $p \vdash E(x_i,x_j)$, then $p \vdash E(x_j,x_i)$ for all $i < j < n$,
 \item if $p \vdash E(x_i,x_j) \wedge E(x_j,x_k)$ for some $i < j < k < n$, then $p \vdash E(x_i,x_k)$, and
 \item if $p \vdash E(x_i,x_j)$ for $i < k < j < n$, then $p \vdash E(x_i,x_k) \wedge E(x_k,x_j)$.
\end{itemize}
That is, $S^\ceq_n(\emptyset)$ are the set of all quantifier-free $n$-types realized by some $<^\A$-ascending sequence of $n$-tuples in $\A$.  Hence $q_1, q_2 \in S^\ceq_n(\emptyset)$.  Now form a graph $\G = (\V,\E)$ with $\V = S^\ceq_n(\emptyset)$ and, for $p_1, p_2 \in S^\ceq_n(\emptyset)$, $\{ p_1, p_2 \} \in \E$ if there exists $0 < j < n$ such that
\begin{itemize}
 \item $p_1(x_0,...,x_{j-1},x_{j+1}, ..., x_{n-1}) = p_2(x_0,...,x_{j-1},x_{j+1}, ..., x_{n-1})$,
 \item $p_1 \vdash \neg E(x_{j-1}, x_{j+1})$ (hence $p_2 \vdash \neg E(x_{j-1}, x_{j+1})$),
 \item $p_1 \vdash E(x_{j-1}, x_j)$ or $p_1 \vdash E(x_j, x_{j+1})$, and
 \item $p_2 \vdash \neg E(x_{j-1},x_j) \wedge \neg E(x_j, x_{j+1})$.
\end{itemize}
That is, $p_1$ thinks $x_j$ is $E$-related to either $x_{j-1}$ or $x_{j+1}$ (but not both, as they are $E$-unrelated) and $p_2$ thinks that $x_j$ is $E$-unrelated to $x_{j-1}$ and $x_{j+1}$.  To each edge $\{ p_1, p_2 \} \in \E$, we can associate a number $1 \le m < n$ of the size of the $E$-class in $p_2$ to which $p_1$ thinks $x_j$ is $E$-related, namely $m := \max \{ m' : m' \le j, p_2 \vdash E(x_{j-m'}, x_{j-1}) \}$ if $p_1 \vdash E(x_{j-1}, x_j)$ and $m := \max \{ m' : m' < n - j, p_2 \vdash E(x_{j+1}, x_{j+m'}) \}$ if $p_1 \vdash E(x_j, x_{j+1})$.  Call this $m(\{ p_1, p_2 \}) = m$.  The reader should check that $\G$ is connected (one can get between any two configurations by pulling apart or adding on a single element to an $E$-class iteratively).  Since $I$ is an indiscernible picture, it induces a vertex $2$-coloring of $\G$, namely $\eta : \V \rightarrow \{0, 1\}$ where $\eta(p) = 1$ if and only if, for some (equivalently any) realization $\aa \models p$ (with $\aa \in A^n$), $\models \theta(I\aa)$.  Hence, $\eta(q_1) = 1$ and $\eta(q_2) = 0$.  In particular, $\eta$ is non-constant.  Now choose an edge $\{ p_1, p_2 \} \in \E$ such that
\begin{itemize}
 \item $\eta(p_1) \neq \eta(p_2)$, and
 \item $m(\{ p_1, p_2 \})$ is minimal such.
\end{itemize}
Since $\G$ is connected and $\eta$ is non-constant, such an edge exists.  Let $j$ witness that $\{ p_1, p_2 \}$ is an edge, without loss suppose $p_1 \vdash E(x_{j-1}, x_j)$ (the argument for $p_1 \vdash E(x_j, x_{j+1})$ is similar), and let $m = m(\{ p_1, p_2\})$.  Without loss of generality, suppose $\eta(p_1) = 0$ and $\eta(p_2) = 1$ (by possibly swapping $\theta$ and $\neg \theta$).  Notice that, by minimal choice of $m(\{ p_1, p_2 \})$, we have in fact that, for all $p \in S^\ceq_n(\emptyset)$, if
\begin{equation}\label{Eq_pp2}
 p(x_0, ..., x_{j-m-1}, x_{j-1}, ..., x_{n-1}) = p_2(x_0, ..., x_{j-m-1}, x_{j-1}, ..., x_{n-1}),
\end{equation}
then $\eta(p) = \eta(p_2) = 1$ (otherwise, we could find a path from $p$ to $p_2$ maintaining this condition of equality, hence contradicting the minimality of $m$).  Similarly, for all $p \in S^\ceq_n(\emptyset)$, if
\begin{equation}\label{Eq_pp1}
 p(x_0, ..., x_{j-m-1}, x_{j-1}, ..., x_{n-1}) = p_1(x_0, ..., x_{j-m-1}, x_{j-1}, ..., x_{n-1}),
\end{equation}
then $\eta(p) = \eta(p_1) = 0$.
Now, for some realization $\aa \models p_1$, let
\[
 \phi(x; y) := \theta( Ia_0, ..., Ia_{j-2}, x, y, Ia_{j+1}, ..., Ia_{n-1} ).
\]
By ultrahomogeneity, there exists an embedding $f : \A \to \A$ such that $a_{j-2} <^\A f(A) <^\A Ia_{j+1}$ and each element of $f(A)$ is $E^\A$ unrelated to both $a_{j-2}$ and $a_{j+1}$.  We claim $\phi$ codes equivalence on $I \circ f$, as desired.


Fix $b_0 <^\A b_1$ from $f(A)$ and let $p(\xx) = \qtp^\A(a_0, ..., a_{j-2}, b_0, b_1, a_{j+1}, ..., a_{n-1})$.  In particular, since $\neg E^\A(b_1, a_{j+1})$ and $m \ge 1$, one of the following holds:
\begin{itemize}
 \item $E^\A(b_0, b_1)$, hence \eqref{Eq_pp1} holds, hence $\eta(p) = 1$, hence $\models \neg \varphi(Ib_0; Ib_1)$; or
 \item $\neg E^\A(b_0, b_1)$, hence \eqref{Eq_pp2} holds, hence $\eta(p) = 0$, hence $\models \varphi(Ib_0; Ib_1)$.
 \end{itemize}
Therefore, $\models \phi(Ib_0; Ib_1)$ if and only if $\neg E^\A(b_0, b_1)$, as desired.
\end{proof}

Recall the general definition of dividing across, given in Definition \ref{Defn_DivideAcross}.  For the specific case of convex equivalence relation indiscernibles, we define dividing vertically.  One could also define the corresponding notion of dividing laterally.

\begin{defn}[Dividing vertically]
Let $\A\models T_\ceq^*$ be a countable model, and let $I:\A\to\UU^y$ be an indiscernible picture. For a formula $\phi(x;y)\in\L(\UU)$, we say that $\phi$ {\em divides vertically across $I$} if there exists an indiscernible sequence $\langle a_i \rangle_{i < \omega}$ on $\A$ with $\neg E^\A(a_0, a_1)$ such that
\[
 \{ \phi(x; Ia_i) \}_{i < \omega} \text{ is inconsistent}.
\]
\end{defn}

\begin{rem}[Coding equivalence implies does not divide vertically]
 Fix $\phi(x;y) \in \L(\UU)$.  If $\phi$ codes equivalence on $I$, then $\phi$ does not divide vertically across $I$.  This is because, for any finite set $B_0 \subseteq A$ of pairwise $E^\A$-unrelated elements, there exists an element $a \in A^n$ such that $a <^\A b$ and $\neg E^\A(a,b)$ for all $b \in B_0$.  Therefore, $\models \phi(Ia; Ib)$ for all $b \in B_0$.  By compactness, for any $B \subseteq A$ of pairwise $E^\A$-unrelated elements, $\{ \phi(x; Ib) \}_{b \in B}$ is consistent.
\end{rem}

\begin{thm}\label{Thm_NTP2}
Let $T$ be a theory with big model $\UU$. The following are equivalent:
\begin{enumerate}
\item $T$ is NTP2.
\item For every indiscernible $T_\ceq$-picture $I:\A\to\UU$, for every formula $\phi(x;y)\in\L(\UU)$, if $\phi(x,y)$ does not divide vertically across $I$, then $\phi(x;y)$ does not divide across $I$.
\end{enumerate}
\end{thm}

\begin{proof}[Proof of 1$\implies$2]
Let $I:\A\to\UU$ be an indiscernible picture of $\A\models T_\ceq^*$ and suppose $\phi(x;y) \in \L(\UU)$ is such that $\phi$ does not divide vertically across $I$.  Suppose $\phi$ divides across $I$ -- that is, we have an indiscernible sequence $\langle a_i \rangle_{i < \omega}$ in $\A$ such that $\left\{\phi(x;Ia_i)\right\}_{i < \omega}$ is inconsistent.  Without loss, suppose $a_i <^\A a_j$ for all $i < j < \omega$.  The only two possibilities for $\langle a_i \rangle_{i < \omega}$ are then
\begin{itemize}
\item $E^\A(a_i, a_j)$ for all $i < j < \omega$, or
\item $\neg E^\A(a_i, a_j)$ for all $i < j < \omega$.
\end{itemize}
Since $\phi$ does not divide vertically across $I$, the second case cannot hold.  Therefore, the first case holds.  However, by compactness, $\{ \phi(x;Ia_i) \}_{i < k}$ is inconsistent for some $k < \omega$.  Since $I$ is an indiscernible picture, for any $E^\A$-related $B_0 \subseteq A$ with $|B_0| = k$, $\{ \phi(x;Ib) \}_{b \in B_0}$ is inconsistent.  Therefore, for any $A' \subseteq A$ so that $(A'; \lhd)$ is isomorphic to $(\omega \times \omega, <_{\mathrm{lex}}, \sim_1)$ (where $(i_0,j_0) \sim_1 (i_1,j_1)$ iff $i_0 = i_1$), we have that $\langle Ia \rangle_{a \in A'}$ is a witness to the fact that $\phi(x;y)$ has $k$-TP2 in $T$.  Therefore, $T$ has TP2.  This completes the proof of 1$\implies$2.
\end{proof}

\begin{proof}[Proof of 2$\implies$1]
 Suppose $T$ has TP2.  
 Let $\varphi(x;y)$ witness 2-TP2 in $T$, and let $\A\models T_\ceq^*$ be the generic model.  Thus there are parameters $B := \{a^i_j : i, j < \omega \}$ in $\UU$ such that:
\begin{itemize}
\item For any finite partial function $f:\omega\to\omega$ with domain $\{i_0<\cdots<i_{n-1}\}$, 
$$\models P^\phi_n\left(b_{f(i_0)}^{i_0},...,b_{f(i_{n-1})}^{i_{n-1}}\right);$$
\item For all $i, j_0, j_1<\omega$, $\models \neg P^\varphi_2(b^i_{j_0}, b^i_{j_1})$.
\end{itemize}
Let $\L^+$ be the expansion of $\L$ (the language of $T$) with one new sort for $\A$ and a new function symbol $h:A\to\UU_y$, and let $\Gamma$ be the set of sentences of $\L^+(A)$ that includes the diagram $\diag(\A)$ of $\A$, the elementary diagram $e\diag(\M)$ of a model $\M\prec \UU$ containing $B$, and the sentences,
$$E^\A(a,a')\cond \neg P_2^\phi(h(a),h(a')),\,\,\bigwedge_{j<k<n}\neg E^\A(a_{j},a_{k})\cond P^\phi_n(h(a_{0}),...,h(a_{n-1}))$$
for all $a,a',a_0,...,a_{n-1}\in A$. Using $B$ to verify finite-satisfiability, it's clear that $\Gamma$ has a model, and from this model, we recover a picture $J:\A\to\UU$ (as a restriction of the interpretation of $h$) such that 
\begin{align}\label{array:thing1}
E^\A(a,a')\implies \neg P_2^\phi(J(a),J(a'))\\
\label{array:thing2}\bigwedge_{j<k<n}\neg E^\A(a_{j},a_{k})\implies P^\phi_n(J(a_{0}),...,J(a_{n-1}))
\end{align}
for all $a,a',a_0,...,a_{n-1}\in A$. Using $J$ to make an \EM-template $A^{<\omega}\to\UU^{<\omega}:\aa\mapsto J\aa$, and with the Modeling Property of $T_\ceq$, we obatin an indiscernible picture $I:\A\to\UU_y$ patterned on $\aa\mapsto J\aa$ -- so also satisfying statements (\ref{array:thing1},\ref{array:thing2}) with $I$ in place of $J$.
In particular,
\begin{align*}
\A \models a_0 < a_1 \wedge E(a_0,a_1) \ \Rightarrow \ \models \neg P^\varphi_2(I(a_0), I(a_1)) \\
\A \models a_0 < a_1 \wedge \neg E(a_0,a_1) \ \Rightarrow \ \models P^\varphi_2(I(a_0),I(a_1))
\end{align*}
Thus the indiscernible picture given by $I$ does not collapse to order.

Moreover, by compactness and \eqref{array:thing1}, for any set of pairwise $E^\A$-unrelated elements $B \subseteq A$, $\{ \varphi(x; Ib) \}_{b \in B}$ is consistent.  Thus, $\phi$ does not divide vertically across $I$.  On the other hand, for any set of pairwise $E^\A$-related elements $B \subseteq A$, $\{ \varphi(x; Ib) \}_{b \in B}$ is inconsistent.  Hence $\phi$ divides across $I$.
\end{proof}

Therefore, we get a dichotomy for theories with NTP2: Either $I$ collapses to order or there exists $\phi(x;y) \in \L(\UU)$ and an embedding $f : \A \to \A$ such that $\phi$ codes equivalence on $I \circ f$.  Hence $\phi$ does not divide vertically across $I \circ f$.  Then, the theorem says that $\phi$ does not divide across $I \circ f$.  In summary:

\begin{cor}
Let $T$ be a theory with big model $\UU$. The following are equivalent:
\begin{enumerate}
\item $T$ is NTP2.
\item For every indiscernible $T_\ceq$-picture $I:\A\to\UU$, either $I$ is an indiscernible $T_1$-picture or, for every formula $\phi(x;y)\in\L(\UU)$ that does not divide vertically across $I$, $\phi(x;y)$ does not divide across $I$.
\end{enumerate}
\end{cor}

\begin{obs} Note that the first condition in (2) above that $I$ be an indiscernible $T_1$-picture is a special case of the second condition, that for every formula $\varphi(x;y) \in \L(\UU)$ that does not divide vertically across $I$, $\varphi(x;y)$ does not divide across $I$.  This is because if $I$ collapses to order, then the type of $I(\aa)$ in $\UU$ is constant for $<^\A$-increasing tuples $\aa$ from $\A$, and so there is no distinction between dividing vertically across $I$ and dividing across $I$.

Also note that the first condition in (2) on its own would not be good enough to imply (1) from the Corollary; look at the case where $\UU$ is a big model of $T_\ceq^*$.  We can argue that the theory is NIP, because basic relations are NIP, and therefore the theory is NTP2.  However, it will not be the case that every indiscernible $T_\ceq^*$-picture is an indiscernible $T_1$-picture.  Consider the example where the identity $J : \A \rightarrow \A$ is used to make the \EM-template on which indiscernible $I : \A \rightarrow \UU$ is patterned.   Clearly $I$ will not be an indiscernible $T_1$-picture, as the equivalence relation is present. 
\end{obs}

One way to think of the (1) $\Rightarrow$ (2) direction of the above Corollary is that, in a TP2 theory, the specific formula witnessing non-collapse to order may not express a distinction between consistency and inconsistency for $\varphi$-types, for some $\varphi$, as that would create a witness to TP2.  Thus we have a partial notion of collapse: In TP2 theories and for indiscernible $T_\ceq$-pictures $I$, there is some restricted set of formulas $\theta(\xx) \in \L(\UU)$ such that if $<^\A$-increasing $\aa, \bb$ give $I(\aa), I(\bb)$ disagreeing on their complete type in $\UU$, it must be disagreement on some $\theta$ from this restricted set.



\end{document}